\documentclass[a4paper,12pt]{amsart}
\usepackage{amsmath,amssymb,amsthm}
\usepackage[left=2cm,right=2cm]{geometry}
\usepackage{enumerate}
\usepackage{mathtools}
\usepackage{mathtools}
\usepackage{hyperref}
\usepackage{verbatim}

\makeatletter
\@namedef{subjclassname@2020}{%
	\textup{2020} Mathematics Subject Classification}
\makeatother

\newtheorem{thm}{Theorem}

\newtheorem{prop}[thm]{Proposition}
\newtheorem{lem}[thm]{Lemma}
\newtheorem{cor}[thm]{Corollary}
\newtheorem{exam}[thm]{Example}

\theoremstyle{definition}
\newtheorem{defn}[thm]{Definition}
\newtheorem{notation}[thm]{Notation}

\theoremstyle{remark}
\newtheorem{rem}{Remark}
\def\C{\mathbb C}
\def\R{\mathbb R}

\def\D{\mathbb D}

\def\eps{\varepsilon}

\def\B{\mathbb{B}}
\def\Bn{\mathbb{B}_n}

\begin{document}
	
	\title[Examples of critically cyclic functions in the Dirichlet spaces of the ball]
	{Examples of critically cyclic functions in the Dirichlet spaces of the ball}
	
	\author{Pouriya Torkinejad Ziarati}
	
	\address{P.T.Ziarati\\
		Institut de Math\'ematiques de Toulouse; UMR5219 \\
		Universit\'e de Toulouse; CNRS \\
		UPS, F-31062 Toulouse Cedex 9, France} 
        \email{pouriya.torkinejad\_ziarati@math.univ-toulouse.fr}

	\begin{abstract}
              In this work, we construct examples of holomorphic functions in $D_2(\B_2)$, the Dirichlet space on $\B_2$, for which there exists a critical index $\alpha_c \in [\frac12,2]$ such that the function is cyclic in $D_\alpha(\B_2)$ if and only if $\alpha \leq \alpha_c$. To this end, we use the notion of \emph{interpolation sets} in smooth ball algebras, as studied by Bruna, Ortega, Chaumat, and Chollet.
        \end{abstract}

	
	
	
	\maketitle
\section{Introduction}
        
        Consider a Banach space of holomorphic functions $\mathcal{A}$ defined on a domain $\Omega$ in $\mathbb{C}^n$. A function $f \in \mathcal{A}$ is called \emph{cyclic} if the linear span of $f$ multiplied by all \emph{multipliers} of $\mathcal{A}$ is dense in $\mathcal{A}$. (Here, $m$ is a multiplier of $\mathcal{A}$ if $mf \in \mathcal{A}$ for all $f \in \mathcal{A}$.) Determining whether a function is cyclic is a space-dependent task. Roughly speaking, on a fixed domain, a more restrictive (larger) norm makes the conditions for cyclicity harder to satisfy. For instance, a function that is not cyclic in the Hardy space $H^2(\mathbb{D})$ is not cyclic in the Dirichlet space $D_1(\D)$ either.

        Another important factor in the study of cyclicity is the zero set of the function. Indeed, let $z_0 \in \overline{\Omega}$ be such that the evaluation functional
        \begin{equation*}
            \Lambda_{z_0} : \mathcal{A} \to \mathbb{C}, \quad f \mapsto f(z_0),
        \end{equation*}
        is bounded. Then any function admitting a zero at $z_0$ fails to be cyclic. In particular, this situation occurs when $z_0 \in \Omega$. Together with the fact that, in all the examples we consider,
functions bounded away from $0$ are easily seen to be cyclic,
this shifts the challenge to the study of boundary behavior of functions. 
        
        In Dirichlet-type spaces on the unit ball $D_\alpha(\Bn)$, which are the subject of this study, one may define for any $f \in D_\alpha(\Bn)$ a non-tangential boundary limit $f^*$ outside a set of $\alpha$-capacity zero \cite{ChalmoukisHartz2024} (precise
definitions of those concepts are given in Section 
\ref{Section2}). It turns out that, for $f$ to be cyclic, the zero set of $f^*$, $\mathcal{Z}(f^*)$, cannot have positive $\alpha$-capacity. A converse to this fact, however, is more subtle and has been the subject of many studies in one and several variables.
        
        In one variable, in 1984, Brown and Shields conjectured that an outer function $f \in D_1(\D)$ is cyclic if and only if $\mathcal{Z}(f^*)$ has zero logarithmic capacity \cite{BS84}. In 2009, El-Fallah, Kellay, and Ransford provided a partial answer for functions continuous up to the boundary with a specific type of zero set~\cite{ElFallah_BS_Conj}, and later extended this result to Dirichlet-type spaces by replacing logarithmic capacity with $\alpha$-capacity~\cite{ELFALLAH2010_Cantor}. Moreover, in harmonically weighted Dirichlet spaces, El-Fallah, Elmadani, and Labghail established a Brown--Shields type result under certain conditions on the weight measure~\cite{ELFALLAH_Dmu_BS}. In all of these results, the zero sets of $A^{\infty}(\D)$—the algebra of holomorphic functions that are $\mathcal{C}^{\infty}$-smooth up to the boundary—known as Beurling–Carleson sets, play an indispensable role. A modern treatment of these sets can be found in the work of Ivrii and Nicolau~\cite{IvriiNicolau_Beurling_Carleson_Set}. 
        
        In several variables, on the unit ball, the concepts of \emph{interpolation sets} and \emph{peak sets} in smooth ball algebras (which are more restrictive than zero sets) may provide a useful tool for studying cyclicity. In~~\cite{Pouriya_SecondPaper} the author used a result of Chaumat and Chollet~\cite{ChaumatChollet1979ensembles} on the peak sets of $A^{\infty}(\Bn)$ to obtain sharp cyclicity conditions for a class of holomorphic functions defined on a neighborhood of the closed unit ball, whenever their zero set is a compact complex-tangential manifold (which, by~\cite{ChaumatChollet1979ensembles}, is a peak set of $A^{\infty}(\Bn)$), thereby providing further positive evidence for the conjecture posed in~\cite[p. 3908]{VavitsasNonCyc} about the 
 cyclicity of polynomials.

        Consider a closed set $E \subset \partial \Bn$ and $\alpha_c \in [0,n]$ such that $E$ has zero $\alpha$-capacity if and only if $\alpha \leq \alpha_c$. Then any function in $D_\alpha(\Bn)$ whose zero set contains $E$ is non-cyclic whenever $\alpha > \alpha_c$. We call a function $f \in D_n(\Bn)$ a \emph{critically cyclic function} for the set $E$ if it has $E$ as its zero set and is cyclic in $D_\alpha(\Bn)$ if and only if $\alpha \leq \alpha_c$.

        Here, we use a sufficiency result due to Bruna and Ortega~\cite{Bruna1986} on the interpolation sets of $A^{m,s}(\Bn)$ for $0 < m + s < \infty$ and one provided by Chaumat and Chollet \cite{Chaumat_Chollet_Hausdorff} in $A^\infty(\Bn)$ to provide examples of critically cyclic functions in Dirichlet-type spaces on the unit ball. When these sets are contained in a transversal curve on the boundary of the unit ball, they satisfy a property called the $\mathcal{K}$-property. In the one-dimensional case, they are also known as \emph{porous sets} and have been studied in~\cite{Thomas_Nicolau_Borichev_Porous}, \cite{Dyn79}.
        
        As the results in~\cite{Bruna1986} and \cite{Chaumat_Chollet_Hausdorff} suggest, we study separately three interesting cases for the boundary zero set, namely the \emph{transversal}, the \emph{complex tangential}, and a case that is neither transversal nor complex tangential but is still \emph{totally real}. Later in Section~{\ref{Section2}}, we will provide precise definitions of these cases.

        In our examples, all constructed in the unit ball of $\C^2$, one observes that the relationship between the Hausdorff dimension of the zero set and the critical index $\alpha_c$ depends on the geometric nature of the set.
        In the transversal case, the zero set has Hausdorff dimension $d$ and the corresponding critical index satisfies $\alpha_c = 2 - d$.
        In the complex tangential case, however, the critical index is given by $\alpha_c = 2 - \frac{d}{2}$.
        We also construct an example by taking a compact subset of a transversal curve, of Hausdorff dimension $d$, and extending it along a complex tangential direction.
        This produces a totally real set of dimension $1 + d$, for which $\alpha_c = \frac{3}{2} - d$.
        
        These apparently different behaviors can be unified by introducing the notion of \emph{Kor\'anyi pseudodistance}
$d_K(z,w)= |1 - \langle z, w \rangle|$
(see Definition \ref{Korany_distance} below)
and subsequently of
        \emph{Kor\'anyi Hausdorff dimension} (Definition~\ref{Korany_dimension}), in terms of which the critical index satisfies the unified relation
        \[
        \alpha_c = 2 - \dim_{K}(E).
        \]
        This is also in agreement with the existing results in the literature for polynomials
        \cite{Kosinski_Vavitsas2023, Vavitsas2, VavitsasNonCyc}
        and holomorphic functions on a neighborhood of the closed unit ball
        \cite{Pouriya_SecondPaper}.

        To state our first result, we 
need the function space $A^k(\Bn) = \mathcal{O}(\Bn) \cap \mathcal{C}^k (\overline{\B}_n)$. By $d_K(z,E)$ we mean the Korányi pseudodistance between $z \in \overline{\B}_n$ and the set $E \subset \partial \Bn$ (see Definition~\ref{Korany_distance}). 

        In the transversal case, the following theorem gives us the necessary tool to prove our first main result:

        \begin{prop}{\cite[Theorem~4.3]{Bruna1986}} \label{MainTool1}
        Let $E$ be a $\mathcal{K}$-set contained in a simple transverse curve $\Gamma$.
        Then there exists $f\in A^1(\Bn)$ such that
        \begin{equation} \label{peak function estimate}
            |f(z)| \asymp d_K(z, E), \qquad z \in \overline{\B}_n.
        \end{equation}
        Moreover, $f$ has this property that $f^k \in A^k(\Bn)$ for $k$ being a positive integer.
        \end{prop}

       We call the function $f$ in~\eqref{peak function estimate} an \emph{extremal function} for the set $E$. 
       
      Let $|A|$ be the one dimensional Lebesgue measure of a 
      measurable set $A$. Let $E\subset \Gamma$ a transverse curve, we define:
 \begin{equation} \label{eq:E_t Leb Measure}
                E_{t} = \{z \in \Gamma: \mathrm{dist}(z,E) \leq t\},
            \end{equation}
    where $\mathrm{dist}$ refers to the Euclidean distance.        
    
    Using the scale of spaces $D_\alpha$ (see
    Definition \ref{dalpha}), our first main result reads:   
        \begin{thm} \label{MainThm1}
Let $E$ be a $\mathcal{K}$-set contained in a simple transverse curve $\Gamma$, and $d \in (0,1)$ be the Hausdorff dimension of $E$. Then $\dim_{K}(E) = d$.

Assume that whenever $t \in \R^+$ is small enough we have $|E_t|  \asymp t^{1 - d}$, and set $\alpha_c = 2 - \dim_{K}(E)$.
Then if  $f$ is an extremal function for the set $E$,
$f$ is cyclic in $D_\alpha(\B_2)$ if and only if $\alpha \leq \alpha_c$.
        \end{thm}

        In the complex tangential case, the following, which is a corollary of Theorem~{II} in \cite{Bruna1986} will serve as the main tool to establish our second main result:

        \begin{prop} \label{MainTool2}
        Let $M$ be a compact $\mathcal{C}^\infty$-smooth complex tangential curve and $E \subseteq M $ be compact. Given $ 0 <\epsilon < \frac{1}{2}$, there exists $f \in \mathcal{O}(\B_2)$ such that $\mathcal{Z}(f) \cap \overline{\B}_2 = E$ and for all $ k$ being a positive integer, we have:
        \begin{equation} \label{eq: Maintool2 estimate1}
             |f(z)| \asymp d_K(z,E)^{1+\epsilon},
        \end{equation}
        and
        \begin{equation} \label{eq: Maintool2 estimate2}
            |R^{k} f(z)| \lesssim d_K(z,E)^{-k+1+\epsilon},
        \end{equation}
        \end{prop}
        where $R f$ denotes the radial derivative of a 
        function $f$ (see Definition~{\ref{Def: Radial Derivative}}).
        
Similarly to the above situation, Let $E\subset M$ a complex tangetial curve, we define: 
\begin{equation}
      E_t = \{z \in M: \mathrm{dist}(z,E) \leq t\}.
\end{equation}

        \begin{thm} \label{MainThm2}
        Let $M$ be a compact $\mathcal{C}^\infty$-smooth complex tangential curve, $E \subseteq M $ be compact,
and $d \in (0,1)$ be the Hausdorff dimension of $E$. 
Then $\dim_{K}(E) = \frac{d}{2}$.

Assume further that $|E_t|\asymp t^{1 - d}$; set
$\alpha_c = 2 - \dim_{K}(E)$, and let $f$ be a function
chosen as in Proposition~{\ref{MainTool2}}.
Then $f$ is cyclic in $D_\alpha(\B_2)$ if and only if $\alpha \leq \alpha_c$.
        \end{thm}

        In the following, which is our third main result, we create examples of critically cyclic functions in Dirichlet-type spaces with a zero set which is not contained in a complex tangential manifold or a transversal curve but is still contained in a totally real manifold. 

    \begin{thm}\label{MainThm3}
        Let $M \subset \partial \B_2$ be a smooth totally real manifold of dimension $2$. 
        Then for any $\kappa \in \left[\tfrac12,\tfrac32\right]$ there exist a compact set $E \subseteq M$ with $\dim_K(E)=\kappa$ and critical index 
        \[
        \alpha_c = 2 - \dim_K(E),
        \]
        and a function $f \in D_2(\B_2)$ such that $\mathcal{Z}(f) \cap \overline{\B}_2 = E$, and $f$ is cyclic in $D_\alpha(\B_2)$ if and only if $\alpha \leq \alpha_c$.
    \end{thm}

The paper is organized as follows. In Section~\ref{Section2}, we provide the preliminary material necessary to follow the rest of the paper, Section~\ref{Section3} presents the proof of Theorem~\ref{MainThm1}, Section~\ref{Section4} is devoted to the proof of Theorem~\ref{MainThm2}, and Section~\ref{Section5} contains the proof of Theorem~\ref{MainThm3}. Finally, we conclude with examples and a discussion in Section~\ref{Section6}.

\section{Preliminary Tools} \label{Section2}
\subsection{Basics and Notations} \label{Subsec2.1}
\begin{notation}
    Let $\Omega \subseteq \mathbb{C}^n$ be a domain. We denote by $\mathcal{O}(\Omega)$ the set of all holomorphic functions \( f : \Omega \to \mathbb{C} \).
\end{notation}
\begin{notation}
    We write $A \lesssim B$ to indicate that there exists a constant $C > 0$ such that $A \leq C B$. If both $A \lesssim B$ and $B \lesssim A$ hold, we write $A \asymp B$. The dependence of the constant $C$ on parameters will be specified whenever necessary.
\end{notation}

\begin{defn}
    Let \( \alpha \in \mathbb{R} \). A function \( f \in \mathcal{O}(\mathbb{B}_n) \), with power series expansion
    \begin{equation*}
        f(z) = \sum_{|k| = 0}^{\infty} a_k z^k, \quad z \in \mathbb{B}_n,
    \end{equation*}
    is said to belong to the Dirichlet-type space \( D_\alpha(\mathbb{B}_n) \) if
    \begin{equation*}
        \|f\|_{D_\alpha(\mathbb{B}_n)}^2 = \|f\|_{\alpha}^2 := \sum_{|k| = 0}^{\infty} (n + |k|)^\alpha \cdot \frac{(n - 1)!}{k! \, (n - 1 + |k|)!} \, |a_k|^2 < \infty,
    \end{equation*}
    where $k$ is a multi-index, \( |k| = k_1 + \cdots + k_n \), \( k! = k_1! \cdots k_n! \), and  \( z^k = z_1^{k_1} \cdots z_n^{k_n} \).
\end{defn}

\begin{defn}
\label{dalpha}
    Let \( \alpha \in \mathbb{R} \). The \emph{multiplier space} \( \mathcal{M}(D_\alpha(\Bn)) \) is the set of all functions \( \varphi \in D_\alpha(\mathbb{B}_n) \) such that for every \( f \in D_\alpha(\mathbb{B}_n) \), the product \( \varphi f \in D_\alpha(\mathbb{B}_n) \).

    The space \( \mathcal{M}(D_\alpha(\Bn)) \) is a Banach algebra when equipped with the operator norm induced by multiplication. In particular, when \( \alpha = 0 \), so that \( D_0(\mathbb{B}_n) = H^2(\mathbb{B}_n) \), we have
    \[
        \mathcal{M}(D_0(\mathbb{B}_n)) = H^\infty(\mathbb{B}_n),
    \]
    the space of bounded holomorphic functions on \( \mathbb{B}_n \).
\end{defn}

\begin{defn}
    $A^{k,s}(\Bn) = \mathcal{O}(\Bn) \cap \mathcal{C}^{k,s}(\overline{\B}_n)$. Also, $A^{k}(\Bn) = A^{k,0}(\Bn)$. 
\end{defn}

\begin{defn} \label{Def: Radial Derivative}
    Given $f \in \mathcal{O}(\Bn)$, the \emph{radial derivative} of $f$, denoted by $Rf$, is defined as follows:
    \begin{equation*}
        R f (z) = \sum_{i = 1}^n z_i \frac{\partial f}{\partial z_i} (z).
    \end{equation*}
\end{defn}

The following Lemma is a direct consequence of the fact that $D_{\alpha}(\Bn) = A^2_{-(\alpha+1)}(\Bn)$ in the sense of the generalization given to the definition of Weighted Bergman Spaces in \cite{ZhuZhao}.
\begin{lem} \label{RLem}
    $f \in D_{\alpha}(\Bn)$ if and only if $R^k f \in D_{\alpha - 2k}(\Bn)$. Also whenever $\alpha<0$ we have
    \begin{equation} \label{EqvBergNorm}
        || f ||^2_\alpha \simeq \int_{\Bn} (1-|z|^2)^{-(\alpha+1)} |f(z)|^2 \, dv(z).
    \end{equation}
\end{lem}
\begin{cor} \label{MultiplierCor}
    Let \( \alpha \in \mathbb{R} \), and let \( k \in \mathbb{N} \) be such that \( \alpha - 2k < 0 \). If $f \in D_\alpha(\Bn)$ and \( R^k f \in H^\infty(\mathbb{B}_n) \), then \( f \in \mathcal{M}(D_\alpha(\Bn)) \).
\end{cor}
\begin{defn}
    A function $f \in D_\alpha(\Bn)$ is called a \emph{cyclic} function if:
    \begin{equation*}
        [f] := \overline{\operatorname{span} \{ p f : p \in \mathbb{C}[z_1,\ldots,z_n] \}} = D_\alpha(\mathbb{B}_n).
    \end{equation*}
\end{defn}

The next theorem provides a practical sufficient condition for the cyclicity of a function. For a detailed proof, one may consult \cite{Knese2019_Aniso_Bidisk}.

\begin{thm} \label{Radial_Dilation}
    Let $f \in D_{\alpha}(\Bn)$ and assume that $f$ does not vanish on $\Bn$. For $0 < r < 1$, define the radial dilation
    \begin{equation*}
        f_r (z) = f(rz).
    \end{equation*}
    If
    \begin{equation*}
        \sup_{0 < r < 1} \left\| \frac{f}{f_r} \right\|_{\alpha} < +\infty,
    \end{equation*}
    then $f$ is cyclic in $D_{\alpha}(\Bn)$.
\end{thm}

\begin{defn} \label{Korany_distance}
The Korányi pseudodistance on $\mathbb{C}^n$ is
defined by
\[
    d_K(z, w) = |1 - \langle z, w \rangle|,
\]
where $\langle z, w \rangle = \sum_{i=1}^n z_i \bar{w}_i$ denotes the standard Hermitian product.
This would be a distance on the unit sphere $\partial\Bn$,
except it only satisfies a quasi triangle inequality 
(up to a multiplicative constant; $d_K^{1/2}$ is truly
a distance, and called Kor\'anyi distance by most authors).
We shall extend the notation to situations where one of the points lies outside $\partial\Bn$.
For $z \in \mathbb{C}^n$ and $r > 0$, we denote by $K(z, r)$ the \emph{Korányi ball}
centered at $z$ with radius $r$.
\end{defn}

\begin{defn} \label{Korany_dimension}
(See, e.g., \cite{Pansu2018}.)
For $s\ge 0$ and $\delta>0$, the $s$--dimensional Kor\'anyi Hausdorff pre--measure
of a set $E\subset \partial\mathbb B_n$ is defined by
\[
\mathcal H^{s}_{K,\delta}(E)
:=
\inf\left\{
\sum_{j} r_j^{\,s}
:
E\subset \bigcup_{j} K(z_j,r_j),
\ \ 0<r_j\le \delta
\right\},
\]
The $s$--dimensional Kor\'anyi Hausdorff measure of $E$ is then
\[
\mathcal H^{s}_{K}(E)
:=
\lim_{\delta\to 0}\mathcal H^{s}_{K,\delta}(E).
\]
The Kor\'anyi Hausdorff dimension of $E$ is defined by
\[
\dim_{K}(E)
:=
\inf\{s\ge 0:\ \mathcal H^{s}_{K}(E)=0\}
=
\sup\{s\ge 0:\ \mathcal H^{s}_{K}(E)=\infty\}.
\]
\end{defn}
Note that our definition of Kor\'anyi Hausdorff dimension yields
a value that is half of what can be found in other references. 
For instance, for us $\dim_{K}(\partial \mathbb B_n)=n$.

\begin{defn}
    For $\zeta \in \partial \Bn$, the complex tangent space to the sphere is
    defined by
    \begin{equation*}
        T_\zeta^{\mathbb C} \partial \Bn
        := \{ v \in \mathbb C^n :  \langle v , \zeta \rangle = 0 \}.
    \end{equation*}    
\end{defn}

\begin{defn}
Let $M \subset \mathbb \partial \Bn$ be a $C^k$ real submanifold.
We say that $M$ is \emph{complex tangential} if, for every $\zeta \in M$,
\begin{equation*}
    T_\zeta M \subset T_\zeta^{\mathbb C}\partial \Bn.
\end{equation*}
\end{defn}

\begin{defn}
    Let $M \subset \mathbb \partial \Bn$ be a $C^k$ real submanifold.
    We say that $M$ is \emph{totally real} if, for every $\zeta \in M$,
    \begin{equation*}
        T_\zeta M \cap i T_\zeta M = \{0\}.
    \end{equation*}
\end{defn}

\begin{defn}
Let $\Gamma \subset \partial \Bn$ be a $C^k$ curve.
We say that $\Gamma$ is \emph{transversal} if, for every $\zeta \in \Gamma$,
\begin{equation*}
    T_\zeta \partial \Bn
    = T_\zeta \Gamma \oplus T_\zeta^{\mathbb C} \partial \Bn.
\end{equation*}
\end{defn}

\begin{rem}
The following facts will be used repeatedly throughout the paper (see, for example, Section~II in \cite{Bruna1986}).

\begin{enumerate}
\item
If $M \subset \partial\Bn$ is a totally real submanifold, then
\begin{equation*}
    \dim_{\mathbb R} M \le n .
\end{equation*}

\item
If $M \subset \partial\Bn$ is a complex tangential submanifold, then
\begin{equation*}
    \dim_{\mathbb R} M \le n-1 .
\end{equation*}

\item
If $M \subset \partial\Bn$ is smooth and complex tangential and $z,w \in M$, then
\begin{equation*}
    d_K(z,w) \asymp |z-w|^2 ,
\end{equation*}
\item
If $\Gamma \subset \partial\Bn$ is a smooth compact transversal curve and $z,w \in \Gamma$,
then
\begin{equation*}
    d_K(z,w) \asymp |z-w| .
\end{equation*}
\end{enumerate}
\end{rem}

\begin{defn} {\cite{Bruna1986}}
A closed set $E \subset \partial\Bn$ is called a \emph{$\mathcal{K}$-set} if there exists a constant $c > 0$
such that, for every Korányi ball $K_r$ of radius $r$,
\begin{equation} \label{eq:K_set condition}
    \int_0^r N_\varepsilon(K_r \cap E)\, d\varepsilon \leq c r,
\end{equation}
where $N_\varepsilon(K_r \cap E)$ denotes the minimal number of Korányi balls
of radius $\varepsilon$ required to cover $K_r \cap E$.
\end{defn}

\begin{defn}
Let $E \subset \mathbb \partial \Bn$ be a Borel set and let $\mathcal P(E)$ denote the set of all
Borel probability measures supported on $E$.  
For $\alpha \in (0,n]$, define the positive kernel
\begin{equation*}
    \mathcal{K}_{\alpha,n}(t) =
    \begin{cases}
        t^{\,\alpha - n}, & 0 < \alpha < n,\\[4pt]
        \log \!\left(\dfrac{e}{t}\right), & \alpha = n.
    \end{cases}
\end{equation*}
For $\mu \in \mathcal P(E)$, the Riesz $\alpha$-energy of $\mu$ is
\begin{equation*}
    I_{\alpha,n}[\mu]
    =
    \iint_{\mathbb \partial \Bn \times \partial \Bn} \mathcal{K}_{\alpha,n}\!\left( d_K(\zeta,\eta) \right)
    \, d\mu(\zeta)\, d\mu(\eta).
\end{equation*}
The \emph{Riesz--$\alpha$ capacity} of $E$ is then defined by
\begin{equation*}    
    \operatorname{cap}_{\alpha,n}(E)
    =
    {\inf \{ I_{\alpha,n}[\mu] : \mu \in \mathcal P(E) \}^{-1}}.
\end{equation*}

The following is a useful tool to determine the non-cyclicity of a function

\begin{thm} {\cite[Theorem~4.4]{sola2014notedirichlettypespacescyclic}} \label{thm:cap_noncyclic}
    Suppose $f \in \mathcal{D}_{\alpha}(\Bn) \cap A(\Bn)$ satisfies $\operatorname{cap}_{\alpha,n}(Z(f) \cap \partial \Bn) > 0$. 
    Then $f$ is not cyclic in $\mathcal{D}_\alpha(\Bn)$.
\end{thm}

\end{defn}

\section{Proof of Theorem~{\ref{MainThm1}}} \label{Section3}

Before proving Theorem~\ref{MainThm1}, we first introduce the following lemma that will serve as a technical tool in its proof.

\begin{lem} {\cite[Exercise~{9.4.1}]{ElFallah2014_book}} \label{lem:book_exercise}
Let $E \subset \mathbb{T}$ be a closed set of measure zero.  
Let $\phi : (0, \pi] \to \mathbb{R}^+$ be a positive, decreasing, and differentiable function.  
Then
\begin{equation*}
    \int_{\mathbb{T}} \phi\big(\mathrm{dist}(\zeta, E)\big)\, |d\zeta|
    = \int_0^{\pi} |\phi'(t)|\, |E_t|\, dt + 2\pi \phi(\pi),
\end{equation*}
where $E_t = \{\zeta \in \mathbb{T} : \mathrm{dist}(\zeta, E) \leq t\}$.
\end{lem}

\begin{proof} [Proof of Theorem~\ref{MainThm1}]
    Let $\alpha_c = 2-d$. Using Lemma~\ref{RLem} and the fact that first-order derivatives of $f$ are bounded, one can easily verify that $f \in D_2(\B_2)$ and in fact, $f \in \mathcal{M}(D_2 (\B_2))$, we will prove that $g = f^2$ is cyclic in $D_{\alpha_c}(\B_2)$, and $f$ is non-cyclic if $\alpha > \alpha_c$.

    Note that cyclicity of $g$ implies cyclicity of $f$, since $f$ being a multiplier implies the following which is a direct consequence of the definition
    \begin{equation*}
        [g] = [f^2] \subseteq [f].
    \end{equation*}

    To prove the cyclicity part it is enough to show that
    \begin{equation} \label{radial_dilation_equation}
        \sup_{r \in (0,1)} \left\| \frac{g}{g_r} \right\|_{\alpha_c} < +\infty.
    \end{equation}
    In this way, we can apply Theorem~\ref{Radial_Dilation} to conclude that $g$ is cyclic.  
    To prove \eqref{radial_dilation_equation}, we use Lemma~\ref{RLem} to show that $\left\| R \frac{g}{g_r} \right\|_{\alpha_c - 2}$ is uniformly bounded.  
    To do so we write:
    \begin{equation} \label{eq: Rg/g_r estimate1}
        \left| R \frac{g}{g_r} \right| = \left|  \frac{g_rRg - g Rg_r}{g^2_r} \right| = \left|  \frac{(g_r - g) Rg + g (Rg - Rg_r)}{g^2_r} \right|.
    \end{equation}
    Let $g = \sum_{k \geq 0} g'_k$ and $f = \sum_{k \geq 0} f'_k$ where $g'_k$'s and $f'_k$'s are $k$-homogeneous polynomials. After applying the triangle inequality to the right hand side of \eqref{eq: Rg/g_r estimate1} we may write
    \begin{equation} \label{eq: Rg/g_r estimate2}
    \begin{aligned}
        \left| R \frac{g}{g_r} \right| &\leq  \left|  \frac{2(f_r+f) f Rf \sum_k(1-r^k)  f'_k}{g^2_r} \right| + \left|  \frac{ g \sum_k(1-r^k) k g'_k}{g^2_r} \right| \\
        &\lesssim \left(\left(\left|\frac{f}{f_r}\right| + \left|\frac{g}{g_r}\right| \right) \left|  \frac{ \sum_k(1+r+\ldots+r^{k-1})  f'_k}{g_r} \right| + \left|\frac{g}{g_r}\right|  \left|  \frac{ \sum_k(1+r+\ldots+r^{k-1}) k g'_k}{g_r} \right| \right) (1-r) \\
        &\lesssim \left(\left(\left|\frac{f}{f_r}\right| + \left|\frac{g}{g_r}\right| \right) \left|  \frac{ \sum_k k f'_k}{g_r} \right| + \left|\frac{g}{g_r}\right|  \left|  \frac{ \sum_k k^2 g'_k}{g_r} \right| \right) (1-r) \\
        &\lesssim (\left(\left|\frac{f}{f_r}\right| + \left|\frac{f}{f_r}\right|^2 \right)      \left|\frac{(1-r)}{f_r^2} \right|.
    \end{aligned}
    \end{equation}
    The last inequality in \eqref{eq: Rg/g_r estimate2} comes from the fact that, $\sum_k k^2 g'_k = R^2g$ and $\sum_k k f'_k = Rf$ are bounded in $\overline{\B}_2$ as $f \in A^1(\B_2)$ and $g \in A^2(\B_2)$, and also replacing $g = f^2$ at the end. Now the claim is that $\frac{f}{f_r} $ has a uniform bound in $H^{\infty}(\B_2)$. Before proceeding to prove this claim note that this together with \eqref{eq: Rg/g_r estimate2} implies:
    \begin{equation} \label{eq:Rg/g_r estimate3}
        \left\| R \frac{g}{g_r} \right\|_{\alpha_c - 2}^2 \lesssim (1-r)^2 \int_{\B_2} \left| \frac{1}{f_r^4(z)}  \right| (1-|z|^2)^{1-\alpha_c}\, \mathrm{d}v(z).
    \end{equation}
    
    To prove the claim, we may write:
    \begin{equation} \label{eq:f/f_r L^inf estimate1}
        \left| \frac{f}{f_r} \right| \leq \left| \frac{f-f_r}{f_r} \right| + 1 \lesssim \left| \frac{\sum_{k} (1-r^k)f_k'}{f_r}\right| +1 \lesssim \left| \frac{1-r}{f_r} \right| +1.
    \end{equation}
    
    Lemma~{5.6} in \cite{Bruna1986} implies that, given $\zeta_0 \in E$ there exists an open neighborhood $U$ of $\zeta_0$ and a coordinate $Z(x_1,x_2,x_3,x_4)$ centered at $\zeta_0$ and valid in $U \cap \overline{\B}_2$, where $Z(x_1,0,0,0)$ is in the inward normal direction to $\partial \B_2$ when $x_1 > 0$, $x_4 \in \mathbb{T}$, $Z(0,x_2,x_3,x_4) \in \partial \Bn$ and $Z(0,0,0,x_4) \in \Gamma$ satisfying
    \begin{equation}
        d_K(Z,E) \asymp x_1 + x_2^2 + x_3^2 + \mathrm{dist}(Z(0,0,0,x_4),E).
    \end{equation}
    Then, for $r > \tfrac{1}{2}$, there exists $c > 0$ such that
    \begin{equation} \label{f/fr eq2}
        |f_r(Z)| \geq c \Big( (1-r) + x_1 + x_2^2 + x_3^2 + \mathrm{dist}(Z(0,0,0,x_4),E) \Big).
    \end{equation}
    Fixing the notation $t = \mathrm{dist(Z(0,0,0,x_4),E)}$, and noting that we can cover $E$ by finitely many such $U$'s settles the claim using \eqref{eq:f/f_r L^inf estimate1}. Moreover, this, together with \eqref{eq:Rg/g_r estimate3} allows us to write: 
    \begin{equation} \label{f/fr eq3}
        \left\| R \frac{g}{g_r} \right\|^2_{\alpha_c - 2}
        \lesssim (1 - r)^2 \int_{-1}^1 \int_{-1}^1 \int_{-1}^1 \int_0^1 
        \frac{x_1^{1 - \alpha_c}}{((1-r) + x_1 + x_2^2 + x_3^2 + t)^{4}} \, dx_1\, dx_2\, dx_3\, dx_4.
    \end{equation}

    Letting \( x_1 = x_1^d \), we may rewrite \eqref{f/fr eq3} as
    \begin{equation}\label{eq:main-estimate}
    \begin{aligned}
    \left\| R \frac{g}{g_r} \right\|_{\alpha_c}^2
    &\lesssim (1 - r)^2 \int_{-1}^1 \int_{-1}^1 \int_{-1}^1 \int_0^1  
    \frac{1}{((1 - r) + x_1^{1/d} + x_2^2 + x_3^2 + t)^{4}} \, dx_1\, dx_2\, dx_3\, dx_4 \\
    &\asymp (1 - r)^2 \int_{-1}^1 \int_{-1}^1 \int_{-1}^1 \int_0^1 
    \frac{1}{((1 - r)^d + x_1  + x_2^{2d} + x_3^{2d} + t^{ d})^{\frac{4}{d}}} \, dx_1\, dx_2\, dx_3\, dx_4 \\
    &\asymp (1 - r)^2 \int_{-1}^1 \int_{-1}^1 \int_{-1}^1
    \frac{1}{((1 - r)^d + x_2^{2 d} + x_3^{2d} + t^{ d})^{\frac{4}{d} - 1}} \, dx_2\, dx_3\, dx_4 \\
    &\asymp (1 - r)^2 \int_{-1}^1 \int_{-1}^1 \int_{-1}^1
    \frac{1}{((1 - r) + x_2^2 + x_3^2 + t)^{\alpha_c + 2}} \, dx_2\, dx_3\, dx_4 \\
    &\asymp (1 - r)^2 \int_{-1}^1
    \frac{1}{((1 - r) + t)^{\alpha_c + 1}} \, dx_4.
    \end{aligned}
    \end{equation}
Let $\gamma: \mathbb{T} \rightarrow \tilde{\Gamma}$, where $\tilde{\Gamma}$ is a closed continuation of $\Gamma$ and $\gamma$ is a parametrization of $\tilde{\Gamma}$. Now define $\tilde{E} = \gamma^{-1}(E)$, the right hand side in \eqref{eq:main-estimate}, can be estimated using Lemma~\ref{lem:book_exercise} and recalling \eqref{eq:E_t Leb Measure} as follows:
\begin{equation}
    \begin{aligned}
        & (1 - r)^2 \int_{-1}^1
    \frac{1}{((1 - r) + t)^{\alpha_c + 1}} \, dx_4 \\
    & \asymp (1 - r)^2 \int_0^{\pi} \frac{1}{((1 - r) + t)^{\alpha_c + 2}}\, |\tilde{E}_t|\, dt \\
    & \asymp (1 - r)^2 \int_0^{\pi} \frac{t^{1-d}}{((1 - r) + t)^{\alpha_c + 2}}\, \, dt \\
    &\asymp (1-r)^{2 - \alpha_c -d} = 1 < + \infty.
    \end{aligned}
\end{equation}
Which implies cyclicity of $f$ in $D_{\alpha_c}(\B_2)$ and consequently for any $D_{\alpha}(\B_2)$ with $\alpha \le \alpha_c$.

To prove non-cyclicity of $f$ in $D_{\alpha}(\B_2)$ whenever $\alpha > \alpha_c$ first note that for $\gamma(s),\gamma(s') \in \Gamma$ we have:
\begin{equation*}
    d_K(\gamma(s),\gamma(s')) \asymp |s - s'|.
\end{equation*}
This implies 
\begin{equation*}
    \dim_{K}(E) = d
\end{equation*}
and
\begin{equation*}
    \operatorname{cap}_{\alpha,2}(E) \asymp \operatorname{cap}_{\alpha-1,1}(\tilde{E}).
\end{equation*}
Now the following, which is a corollary of Theorem~1 in Chapter~IV of~\cite{carleson1967selected}, completes the proof (see also Theorem~4.3 in~\cite{ELFALLAH2010_Cantor}):
\begin{equation*}
    \operatorname{cap}_{\alpha-1,1}(\tilde{E}) > 0 
    \quad \Longleftrightarrow \quad
    \int_{0}^{\pi} \frac{dt}{t^{2 -\alpha} |E_{t}|} < \infty ,
\end{equation*}
where the right hand side condition holds exactly when $\alpha > \alpha_c = 2-d$.

\end{proof}

\section{Proof of Theorem~{\ref{MainThm2}}} \label{Section4}

In order to prove Proposition~{\ref{MainTool2}}, we will first introduce the following lemma
\begin{lem} \label{lem:for MainTool2}
    Let $E \subseteq [0,1]$ be compact; then given $\epsilon > 0$ there exists $f \in \mathcal{C}^{2,\epsilon}\left([0,1]\right)$ such that
    \begin{equation} \label{eq:f dist to E property}
        f(x) \asymp \mathrm{dist}(x,E)^{2+\epsilon}.
    \end{equation}
\end{lem}
\begin{proof}
    We define $\phi:\R \rightarrow \R$ as follows:
    \begin{equation*}
        \phi(x) =
    \begin{cases}
        x(1-x), & x\in [0,1],\\[4pt]
        0, & o.w.
    \end{cases}
    \end{equation*}
    Let $[0,1]\setminus E = \bigsqcup_{k \in \mathbb{N}} I_k$, where $I_k = (a_k,b_k)$ and, $|I_k| = b_k-a_k$, we may consider $\phi_k(x) = |I_k|\phi(\frac{x-a_k}{|I_k|})$. By definition we have $\mathrm{supp}\phi_k = \bar{I}_k$ and
    \begin{equation*}
        \phi_k(x) \asymp \operatorname{dist}(x,E), \qquad x \in I_k.
    \end{equation*}
    Now, it is straightforward to verify $f = \sum_k \phi_k^{2 + \epsilon} \in \mathcal{C}^{2,\epsilon}([0,1])$ and satisfies \eqref{eq:f dist to E property}. 
\end{proof}

\begin{proof}[Proof of Proposition~\ref{MainTool2}]
    Using Lemma~{\ref{lem:for MainTool2}} we may create $\psi \in \mathcal{C}^{2,2\epsilon}(M)$ such that
    \begin{equation*}
        \psi \asymp \mathrm{dist}^{2+2\epsilon}(.,E).
    \end{equation*}
    Given $\zeta_0 \in M$ let the coordinate map $Z(x_1,x_2,x_3,x_4)$ valid on a neighborhood $U$ of $\zeta_0$ be such that for all $\zeta \in U \cap M$
    \begin{enumerate} [(i)]
        \item $\left( \frac{\partial}{\partial x_3} (\zeta) , \frac{\partial}{\partial x_4} (\zeta) \right)$ span the complex tangent plane. 
        \item $\frac{\partial}{\partial x_4}(\zeta)$ is the tangent vector along $M$ at $\zeta$.
        \item $\frac{\partial}{\partial x_1}(\zeta)$ is the inward orthogonal vector to the tangent space at $\zeta$.
        \item $\frac{\partial}{\partial x_2}(\zeta)$ denotes the complex normal direction at $\zeta$.
    \end{enumerate}
    Now we may define a first order non-isotropic jet $F$ on $U \cap M$, i.e a function and its pre-determined derivatives with respect to $x_1$ and $x_2$ up to order $1$ and with respect to $x_3$ and $x_4$ up to order $2$. To this end for $\zeta \in U \cap M$ we write
    \begin{equation} \label{eq:Jet values}
    \begin{aligned}
    &F(\zeta) = \psi(\zeta), \\[4pt]
    &\dfrac{\partial F}{\partial x_i}(\zeta) = 0, \quad i=1,2, \\[4pt]
    &\dfrac{\partial F}{\partial x_4}(\zeta)
    = -\, i\, \dfrac{\partial F}{\partial x_3}(\zeta)
    = \psi'(\zeta), \\[4pt]
    &\dfrac{\partial^2 F}{\partial x_4^2}(\zeta)
    = - \dfrac{\partial^2 F}{\partial x_3^2}(\zeta)
    = -\, i\, \dfrac{\partial^2 F}{\partial x_3 \partial x_4}(\zeta)
    = \psi''(\zeta).
    \end{aligned}
    \end{equation}
    In fact, a similar reasoning to the discussion in the third paragraph of page~{568} in \cite{Bruna1986} allows us to extend the jet $F$ on $M$ in a well-defined fashion. Let $TF_\zeta$ be the Taylor polynomial generated using $F$ at $\zeta$, in such a case one may observe that
    \begin{equation*}
        |F_X(\zeta)-XTF_{\zeta}(z)| \lesssim d_K(z,\zeta)^{1 + \epsilon - w(X)},
    \end{equation*}
    where $w(X)$ refers to the weight associated to a composition of differential operators $X = X_1 \ldots X_k$, for example the radial derivative has weight $1$ and derivation along the complex tangential directions have weight $\frac{1}{2}$. Also, by $F_X$, we mean the element of $F$ which corresponds to the value of pre-determined $X$ derivative. For more discussion in this regard one may consult the introduction part in \cite{Bruna1986}.
    
    Theorem~{II} in \cite{Bruna1986} states that $M$ is an interpolation set for $A^{1,\epsilon}(\B_2)$. This guarantees that there exists $f \in A^{1,\epsilon}(\B_2)$ such that for all $\zeta \in M$, $f$ and its first order non-isotropic derivatives match with $F$. In particular, for $\zeta \in E$, the fact that $\psi(\zeta) = \psi'(\zeta) = \psi''(\zeta) = 0$ implies
    \begin{equation*} \label{eq: Maintool2 estimate_1}
        |f(z)| \lesssim d_K(z,E)^{1+\epsilon},
    \end{equation*}
    and,
    \begin{equation*} \label{eq: Maintool2 estimate_2}
        |X f(z)| \lesssim d_K(z,E)^{-w(X)+1+\epsilon},
    \end{equation*}
     Also, there exists $h \in A^{\infty}(\B_2)$ such that $|h(z)| \asymp d_K(z,M)$ and $\Re h >0$ on $\B_2$ (see Theorem~{6.2} in \cite{Bruna1986}). Now, perhaps by replacing $f$ with $c  h^{1+\epsilon} + f$ when $c>0$ is large enough, we may claim that
    \begin{equation*}
        |f(z)| \gtrsim d_K(z,E)^{1+\epsilon}.
    \end{equation*}
    To verify the claim, let $c_0>0$ be such that
    \[
        |f(z)| \le c_0 d_K (z,E)^{1+\eps} \mbox{ and } |Xf(z)| \le c_0 d_K (z,E)^{1+\eps-w(X)}.
    \]
    Now, choosing $h$ such that $h(z)\ge d_K(z,M)$,
    replace $f$ with $f_1(z):= f(z) + c_1 h(z)^{1+\eps}$, with $c_1 = 2 c_0 c_2^{1+\eps}$, $c_2>0$ to be chosen later. We claim that for $c_2$ large enough, there exists $c_3 >0$ such that $|f_1(z)| \ge c_3 d_K (z,E)^{1+\eps}$. 
    
    If we assume $d_K (z,E)\le c_2 d_K(z,M)$, the choice of $c_1$ ensures that $|f_1(z)|\ge \frac12 c_1 |h(z)|^{1+\eps} \ge c_0 d_K (z,E)^{1+\eps}$.

    Otherwise if $d_K (z,E)\ge c_2 d_K(z,M)$, choose a point $\pi(z)\in M$ such that $d_K(z,M)= d_K(z,\pi(z))$.  
    For $z\in \overline \B^2$ close enough to $M$, we can choose it so that for any $w\in [z;\pi(z)]$, $d_K(w,E)\le 2 d_K(z,E)$,
    so that 
    \[
        \left| f(z) - f(\pi(z)) \right| \le \left( \sup_{[z;\pi(z)]} |\nabla f | \right) d_K(z,\pi(z)) \le 2 c_0 d_K(z,E)^\eps d_K(z,M)
        \le 2c_0 c_2^{-1} d_K (z,E)^{1+\eps}.
    \]
    But 
    \[
        f(\pi(z)) = \psi(\pi(z))\ge c_\psi d_K (\pi(z),E)^{1+\eps} \ge 2^{-(1+\eps)} c_\psi d_K (z,E)^{1+\eps} \ge 2 | f(z) - f(\pi(z))| 
    \]
    for $c_2$ large enough, so
    that $| f(z)| \ge \frac12  c_\psi d_K (z,E)^{1+\eps} $ and $|\arg(f(z))| \le \frac\pi6$.

    On the other hand, since $\Re h(w)>0$ for $w\in \B^2$, we have $|\arg(h(z)^{1+\eps})|\le \frac\pi2 (1+\eps)$, so that 
    $|\arg(f(z))-\arg(h(z)^{1+\eps})|\le \frac\pi6 + \frac{3\pi}4$. We use the elementary fact that if two complex numbers 
    verify $| \arg \zeta - \arg \xi | \le \pi - \theta_0$, $\theta_0>0$, then $|\zeta + \xi | \ge (\sin \frac{\theta_0}2) (|\zeta| + |\xi |)$,
    and deduce $|f_1(z)| \ge (\sin \frac\pi{24}) | f(z)| \gtrsim d_K (z,E)^{1+\eps} $.  The claim is proved. 
\end{proof}

\begin{proof} [Proof of Theorem~\ref{MainThm2}]
    Let $f$ be as in Proposition~{\ref{MainTool2}}, using \eqref{eq: Maintool2 estimate1}, \eqref{eq: Maintool2 estimate2} and calculations along the lines of those which led to \eqref{eq:Rg/g_r estimate3} one may show that

    \begin{equation} \label{eq: Rg/g_r estimate2}
    \begin{aligned}
        \left| R \frac{f}{f_r} \right|(z) &\leq  \left|  \frac{(f-f_r) Rf}{f^2_r} \right|(z) + \left|  \frac{ f (Rf - Rf_r) }{f^2_r} \right|(z) \\
        &\lesssim \left|  \frac{(1-r) (Rf)^2}{f^2_r} \right|(z) + \left|  \frac{ (1-r) f R^2f }{f^2_r} \right|(z) \\
        &\lesssim \left|\frac{1-r}{f^2_r(z)}\right| d_K(z,E)^{2 \epsilon}    \\
        &\lesssim \frac{1-r}{d_K(z,E)^2},
    \end{aligned}
    \end{equation}
    and
    \begin{equation} \label{eq: shabih2}
        \left\| R \frac{f}{f_r} \right\|_{\alpha - 2}^2 \lesssim (1-r)^2 \int_{\B_2}  \frac{1}{d_K(rz,E)^{4}}   (1-|z|^2)^{1-\alpha}\, \mathrm{d}v(z).
    \end{equation}
    Also, as before, for $r > \frac{1}{2}$ and $\zeta_0 \in E$ there exists a neighborhood $U$ in $\C^2$ of $\zeta_0$ and a coordinate $Z(x_1,x_2,x_3,x_4)$ valid in $U$ such that for all $z \in \overline{\B}_2 \cap U$ we have
    \begin{equation*}
        d_K(rz,E) \gtrsim \left((1-r)x_1 +|x_2| + x_3^2 + \mathrm{dist}^2(Z(0,0,0,x_4),E)\right).
    \end{equation*}
    Letting $t = \mathrm{dist}(Z(0,0,0,x_4),E)$, $\alpha_c = 2 -\frac{d}{2}$ and noting that we may cover $E$ by finitely many such $U$'s allows us to write
    \begin{equation}
        \begin{aligned}
            \left\| R \frac{f}{f_r} \right\|^2_{\alpha_c - 2}
        &\lesssim (1 - r)^2 \int_{-1}^1 \int_{-1}^1 \int_{-1}^1 \int_0^1 
        \frac{x_1^{1 - \alpha_c}}{((1-r) + x_1 + |x_2| + x_3^2 + t^2)^{4}} \, dx_1\, dx_2\, dx_3\, dx_4 \\
        &\asymp (1 - r)^2 \int_{-1}^1 \int_{-1}^1 \int_{0}^1  
        \frac{1}{((1-r) + x_2 + x_3^2 + t^2)^{2 + \alpha_c}} \, dx_2\, dx_3\, dx_4 \\
        &\asymp (1 - r)^2 \int_{-1}^1 \int_{-1}^1  
        \frac{1}{((1-r) + x_3^2 + t^2)^{1 + \alpha_c}} \, dx_3\, dx_4 \\
        &\asymp (1 - r)^2 \int_{-1}^1  
        \frac{1}{((1-r) + t^2)^{\frac{1}{2} + \alpha_c}} \, dx_4 \\
        &\asymp (1 - r)^2 \int_{-1}^1  
        \frac{t^{2-d}}{((1-r) + t^2)^{\frac{5-d}{2}}} \, dt \asymp 1 < +\infty. \\
        \end{aligned}
    \end{equation}
    Now let $\alpha > \alpha_c$, and $\gamma : [0,1] \rightarrow M$ be a diffeomorphism for $s,s' \in [0,1]$ we have
    \begin{equation*}
        d_K(\gamma(s),\gamma(s')) \asymp |s-s'|^2,
    \end{equation*}
    which implies
    \begin{equation*}
       \dim_{K}(E) = \frac{d}{2}
    \end{equation*}
    and
    \begin{equation*}
            \operatorname{cap}_{\alpha,2}(E) \asymp \operatorname{cap}_{2\alpha-3,1}(\exp{\left(2\pi i \gamma^{-1}(E)\right)}).
    \end{equation*}
    Now similar to Theorem~{\ref{MainThm1}} we have
    \begin{equation*}
        \operatorname{cap}_{\alpha,2}({E}) > 0 
        \quad \Longleftrightarrow \quad
        \int_{0}^{1} \frac{dt}{t^{4 -2\alpha} |E_{t}|} < \infty .
    \end{equation*}
    so $\operatorname{cap}_{\alpha,2}(E)> 0$ and we are done.
\end{proof}

\section{Proof of Theorem~{\ref{MainThm3}}} \label{Section5}
    
    The following, which is a rather direct corollary of Proposition~{17} in \cite{Chaumat_Chollet_Hausdorff}, will allow us to prove our third main result. 

        \begin{prop} \label{MainTool3}
            Let $\gamma : (-1,1) \rightarrow \partial \B_2$ be a simple transversal curve. Let $I \subset (-1,1)$ be a closed interval such that there exists a coordinate mapping $Z : [0,1) \times (-1,1)^2 \times I \rightarrow \overline{\B}_2$ with the following properties:
            \begin{enumerate}
                \item $Z (0,0,0,s) = \gamma(s)$.
                \item $Z(0,x_2,x_3,s) \in \partial \B_2$.
                \item $\frac{\partial}{\partial x_1}$ corresponds to normal inward direction to $\B_2$. 
                \item $(\frac{\partial}{\partial x_2} (\zeta),\frac{\partial}{\partial x_3} (\zeta))$ span the complex tangential plane at $\zeta \in \gamma((-1,1))$.
            \end{enumerate}
            Let $\tilde{E} \subset \gamma(I)$ be a $\mathcal{K}$-set and define
            \begin{equation} \label{eq: mainthm3 set creation}
                E = \left\{\, Z(0,x_2,0,s) : x_2 \in [-\frac{1}{2},\frac{1}{2}],\ \gamma(s) \in \tilde{E} \,\right\}.
            \end{equation}
            Then there exists $\psi \in \mathcal{O}(\B_2)$ with the following properties:
            \begin{enumerate}
                \item $\Re \psi (z) \asymp \log(\dfrac{1}{d_K(z,E)}) + O(1)$.
                \item $\left| D^{\beta} \psi \right|(z) \lesssim \dfrac{1}{d_K(z,E)^{|\beta|}}$ where $D^{\beta}$ is a partial derivative with order $|\beta|$.
            \end{enumerate}
        \end{prop}
    \begin{proof}
        Given $s \in I$, we define a compact complex tangential curve $M_s \subset \partial \B_2$ in the following form:
        \begin{equation*}
            M_s = \left\{ Z(0,x_2,0,s): x_2 \in [-\frac{1}{2},\frac{1}{2}] \right\}.
        \end{equation*}
        Then by Lemma~{15} in \cite{Chaumat_Chollet_Hausdorff}, there exists a $\mathcal{C}^\infty$ function $(s,z) \mapsto H(s,z)$ which is holomorphic in $z \in \B_2$ and satisfies the following:
        \begin{equation} \label{eq:H_estimate}
            |H(s,z)| \asymp d_K(z,M_s).
        \end{equation}
        For a positive integer $k$ let $N_k$ be a minimal number of arcs in $\gamma(I)$ of length $2^{-k}$ needed to cover $\tilde{E}$. Then by \eqref{eq:K_set condition} we have:
        \begin{equation} \label{eq:N_k sum}
            \sum_{k = k_0}^{\infty} 2^{-k} N_k \lesssim 2^{-k_0}.
        \end{equation}
        Let $\gamma(s_{jk})$ $j = 1,\ldots,N_k$ be the points in $\tilde{E}$ which realize this covering by the $2^{-k}$ length arcs we will define $\psi$ as follows:
        \begin{equation}
            \psi(z) = \sum_{k = 0}^\infty \sum_{j = 1}^{N_k} \frac{2^{-k}}{2^{-k}+ H(s_{jk},z)}.
        \end{equation}
        The estimates on $\Re \psi$ are the same as the ones in Proposition~{17} in \cite{Chaumat_Chollet_Hausdorff} (see also \cite{Chaumat_Chollet_Division_Prop}). In order to estimate $|D^{\alpha} \psi|$, $|\alpha| \geq 1$, first note that $H$ and all of its derivatives with respect to $z$ are bounded above. This implies
        \begin{equation*}
            |D^{\beta} \psi(z)| \lesssim  \sum_{k = 0}^\infty \sum_{j = 1}^{N_k} \frac{2^{-k}}{\left|2^{-k}+ H(s_{jk},z) \right|^{|\beta|+1}}. 
        \end{equation*}
        For a fixed $k$ let $z$ be close enough to $E$ such that there are at most two $\gamma(s_{jk})$ which are the closest to $z$. In such a case, since $\gamma(s_{jk}) \in \tilde{E}$ we have $d_K(z,E) + j 2^{-k} \lesssim d_K(z,M_{s_{jk}})$ for $j = 1,\ldots, N_k$ as the $\gamma(s_{jk})$'s are at least $2^{-k}$ distanced from each other. So by \eqref{eq:H_estimate} we have
        \begin{equation*}
            |D^{\beta} \psi(z)| \lesssim  \sum_{k = 0}^\infty \sum_{j = 1}^{N_k} \frac{2^{-k}}{\left(2^{-k} j+ d_K(z,E) \right)^{|\beta|+1}}.
        \end{equation*}
        Let $k_0$ be such that $2^{-k_0} \leq d_K(z,E) \leq 2^{-k_0+1}$. Hence, we obtain
        \begin{equation} \label{eq:sum k_geq_k_0}
            \sum_{k = k_0}^\infty \sum_{j = 1}^{N_k} \frac{2^{-k}}{\left(2^{-k} j+ d_K(z,E) \right)^{|\beta|+1}} \leq \sum_{k = k_0}^\infty \frac{2^{-k} N_k}{ d_K(z,E) ^{|\beta|+1}} \lesssim \frac{1}{ d_K(z,E) ^{|\beta|}},
        \end{equation}
        where the last inequality comes from \eqref{eq:N_k sum}. The remaining terms are estimated by
        \begin{equation} \label{eq:sum k_leq_k_0}
            \sum_{k = 0}^{k_0} \sum_{j = 1}^{N_k} \frac{2^{-k}}{\left(2^{-k} j+ d_K(z,E) \right)^{|\beta|+1}} \leq \sum_{k = 0}^{k_0} \sum_{j = 1}^{N_k} \frac{2^{k|\beta|}}{ j ^{|\beta|+1}} \lesssim \sum_{k = 0}^{k_0} 2^{k|\beta|} \asymp \frac{1}{ d_K(z,E) ^{|\beta|}}. 
        \end{equation}
        Note that \eqref{eq:sum k_geq_k_0} and \eqref{eq:sum k_leq_k_0} together finish the proof.
    \end{proof}
    \begin{proof} [Proof of Theorem~\ref{MainThm3}]
        Let $\psi$ and $E$ be as in Proposition~\ref{MainTool3}, moreover assume that
        \begin{equation*}
            |\tilde{E}_t| = \left| \{ \gamma(s): d_K(\gamma(s),\tilde{E}) \leq t \} \right| \asymp t^{1-d},
        \end{equation*}
        where $d \in [0,1]$ is the Hausdorff dimension of $\tilde{E}$. We define $f = \exp{(-\psi)}$. By raising $f$ to a power, if needed, we may assume that $f \in D_2(\B_2)$ and $R^2 f \in H^{\infty}(\B_2)$. Now take $\alpha_c = \frac{3}{2} - d$, in the same way we obtained \eqref{eq: shabih2} we can write
        \begin{equation}
            \left\| R \frac{f}{f_r} \right\|_{\alpha_c - 2}^2 \lesssim (1-r)^2 \int_{\B_2}  \frac{1}{d_K(rz,E)^4}  (1-|z|^2)^{1-\alpha_c}\, \mathrm{d}v(z),
        \end{equation}
        and, 
        \begin{equation}
            d_K(rz,E) \gtrsim  \left( (1-r) + x_1 + x_3^2 + \mathrm{dist}(Z(0,0,0,s),\tilde{E}) \right).
        \end{equation}
        Letting $t = d_K(Z(0,0,0,s),\tilde{E})$, we may write
    \begin{equation}
        \begin{aligned}
        \left\| R \frac{f}{f_r} \right\|^2_{\alpha_c - 2}
        &\lesssim (1 - r)^2 \int_{-1}^1 \int_{-1}^1 \int_{-1}^1 \int_0^1 
        \frac{x_1^{1 - \alpha_c}}{((1-r) + x_1 + x_3^2 + t)^{4}} \, dx_1\, dx_2\, dx_3\, ds.\\ 
    &\asymp (1 - r)^2 \int_{-1}^1 \int_{-1}^1 
    \frac{1}{((1 - r) + x_3^2 + t)^{\alpha_c + 2}} \, dx_3\, ds \\
    &\asymp (1 - r)^2  \int_{-1}^1
    \frac{1}{((1 - r) + t)^{\alpha_c + \frac{3}{2}}} \, ds \\
    &= (1 - r)^2 \int_{-1}^1
    \frac{1}{((1 - r) + t)^{3-d}} \, ds \\
    & \asymp (1 - r)^2 \int_0^{1} \frac{t^{1-d}}{((1 - r) + t)^{4-d}}\, \, dt\ \asymp 1 < \infty, \\
    \end{aligned}
    \end{equation}
    where, in the estimates on the last line, we applied Lemma~\ref{lem:book_exercise}.
    
    To prove the non-cyclicity part, assume that $\alpha > \frac{3}{2} - d$. Let $\mu_1$ be the normalized $d$-dimensional Hausdorff measure restricted to $\gamma^{-1}(\tilde{E})$ and $\mu_2$ be the Lebesgue measure on $[-\frac{1}{2}, \frac{1}{2}]$. We set $\mu$ to be the push-forward measure of $\mu_1 \times \mu_2$ with respect to the mapping $Z$. Therefore, $\mu$ is a probability measure on $E$, and for $\zeta \in E$ the following holds
    \begin{equation*}
        \mu (K(\zeta,r) \cap E) \asymp r^{\frac{1}{2}+d}.
    \end{equation*}
    Therefore,
    \begin{equation*}
        \dim_{K}(E) = \frac{1}{2} + d.
    \end{equation*}
    A straightforward application of the Layer Cake Representation formula(see \cite[p.~26]{LiebLoss}) and Fubini's theorem allows us to write
    \begin{equation*}
    \begin{aligned}
        I_{\alpha,2}[\mu] & = \iint_{\mathbb \partial \B_2 \times \partial \B_2} \mathcal{K}_{\alpha,2}\!\left( d_K(\zeta,\eta) \right) \, d\mu(\zeta)\, d\mu(\eta) \\
        & = \int_{1}^\infty \mu \times \mu (\{ (\zeta,\eta) \in E \times E : \frac{1}{d_K(\zeta,\eta)^{2-\alpha} } \geq r \} ) \, dr + 1 \\
        & \asymp \int_{0}^1 \mu \times \mu (\{ (\zeta,\eta) \in E \times E : d_K(\zeta,\eta)  \leq r \}) \frac{1}{r^{3-\alpha}} \, dr +1 \\
        & = \int_{0}^1 \int_{E} \mu (K(\eta,r)\cap E) \frac{1}{r^{3-\alpha}} \,d\mu(\eta) \, dr +1 \\
        & \asymp \int_{0}^1 \frac{\, dr}{r^{\frac{5}{2}-\alpha - d}} + 1 < \infty,
    \end{aligned}
    \end{equation*}
    which implies that $\operatorname{cap}_{\alpha,2}(E) > 0$ and therefore $f$ is not cyclic in $D_\alpha(\B_2)$.
    \end{proof}
\section{Examples and Discussion} \label{Section6}

In this section, we first provide examples of sets that satisfy the conditions of our main results.

\begin{defn}
    A \emph{Cantor subset} of $\mathbb{T}$ with constant dissection ratio is a closed set
    $E$ obtained by repeatedly removing open arcs as follows.  
    Set $E_0 \subseteq \mathbb{T}$ a closed arc, and for each $n \ge 1$ let $E_n$ be the union of finitely many disjoint closed arcs
    obtained from $E_{n-1}$ by removing $2^{n-1}$ many disjoint open arcs of length $l_n$.  
    Such that
    \begin{equation*}
        \sum_{n=1}^{\infty} 2^{n-1} l_n = |E_0|,
    \end{equation*}
     and,
    \begin{equation*}
        \lambda = \frac{l_{n+1}}{l_n} < \frac{1}{2}.
    \end{equation*}
\end{defn}

In such a case, it is easy to verify that $|E_t| \asymp t^{1-d}$, where $d = \frac{\mathrm{log} 2}{\mathrm{log}(1/\lambda)}$ is the Hausdorff dimension of $E$.

\begin{lem}{\cite[Lemma~4.1]{Bruna1986}} \label{lem: K-set condition}
Let $\Gamma$ be a simple transverse curve in $\partial \Bn$, parameterized by $\gamma : \mathbb{T} \to \Gamma$ and let $E \subset \Gamma$ be closed. And $\tilde{E} = \gamma^{-1}(E)$.

The following are equivalent:
\begin{enumerate}[(i)]
    \item $E$ is a $\mathcal{K}$-set.
    \item There exists $c > 0$ such that for all intervals $J \subset \mathbb{T}$ and $|J| = r$,
    \begin{equation} \label{eq:K-set Circle}    
        \int_0^r N_\varepsilon(J \cap \tilde{E})\, d\varepsilon \leq c r.
    \end{equation}
\end{enumerate}
\end{lem}

\begin{exam}
Let $I \subsetneq \mathbb T$ be a closed interval, and let $\tilde E \subset I$
be a Cantor set with Hausdorff dimension $d\in(0,1)$.
Let $\Gamma \subset \partial\mathbb B_2$ be a simple smooth closed transversal curve,
and let $\gamma:\mathbb T\to\Gamma$ be a parametrization.
Set
\begin{equation*}
    E := \gamma(\tilde E)\subset\partial\mathbb B_2 .
\end{equation*}
\end{exam}

Then $E$ is contained in a transversal curve and has Hausdorff dimension $d$. Moreover, by Lemma~{\ref{lem: K-set condition}}, to verify that $E$ is a $\mathcal{K}$-set it suffices to see that $\tilde{E}$ is a porous set (see section~{7} in \cite{Noell_peakset_survey}).
Therefore $E$ satisfies the assumptions of Theorem~\ref{MainThm1} and can be used to create sets satisfying conditions of Theorem~{\ref{MainThm3}} using \eqref{eq: mainthm3 set creation}. 

\begin{exam}
Let $I \subseteq \mathbb [0,1]$ be a closed interval, and let $\tilde E \subset I$
be a Cantor set with Hausdorff dimension $d\in(0,1)$.
Let $\Gamma \subset \partial\mathbb B_2$ be a smooth complex tangential curve,
and let $\gamma:\mathbb [0,1]\to\Gamma$ be a parametrization.
Define
\begin{equation*}
    E := \gamma(\tilde E)\subset\partial\mathbb B_2 .
\end{equation*}
\end{exam}
Then $E$ trivially satisfies conditions of Theorem~{\ref{MainThm2}}.

We conclude the paper, by mentioning that, the proofs for Theorem~{\ref{MainThm1}} and Theorem~{\ref{MainThm3}} can be easily extended for any dimension $n \geq 2$. This allows us to create examples of critically cyclic functions for any $\alpha_c \in [\frac{n-1}{2}, n]$. However, an analog of Theorem~{\ref{MainThm2}} for $n \geq 2$ would be more technically demanding. It is also useful to note that existence of peak sets for $A(\Bn)$ having Hausdorff dimension $2n-1$, might be promising to search for critically cyclic functions whenever $\alpha_c < \frac{n-1}{2}$  \cite{Henriksen1982}.

\section*{Acknowledgements}
The author is deeply grateful to Łukasz Kosiński and Pascal Thomas, his PhD advisors, for their guidance and constant support during the preparation of this work. He especially thanks Pascal Thomas for his major contribution to the development of the main ideas, and Łukasz Kosiński for his numerous insightful comments and continued encouragement. He also thanks Vincent Feuvrier for useful mathematical discussions. The author also thanks the Institut de Mathématiques de Toulouse and the Department of Mathematics at the Jagiellonian University for providing excellent working conditions.

\bibliographystyle{plain} 

\end{document}